\documentclass[12pt]{amsart}
\usepackage{amsmath, amsthm, amssymb}
\usepackage{tikz}
\usepackage{graphicx}
\usepackage{amsfonts}
\usepackage{pgfplots}
\usepackage{comment}
\usepackage{esint}
\usepackage{cancel}
\usepackage{eucal}
\usepackage{enumerate}
\usepackage{xcolor}
\usepackage[colorlinks=true, pdfstartview=FitV, linkcolor=blue, citecolor=blue, urlcolor=blue]{hyperref}

\usepackage{graphicx} 
\usepackage{color}

\usepackage{marginnote}

\allowdisplaybreaks[4]
\numberwithin{equation}{section}

\newcommand{\ban}{\begin{eqnarray*}}
\newcommand{\ean}{\end{eqnarray*}}
\newcommand{\ra}{\rightarrow}
\newcommand{\be}{\begin{equation}}
\newcommand{\ee}{\end{equation}}
\newcommand{\ba}{\begin{eqnarray}}
\newcommand{\ea}{\end{eqnarray}}

 \newcommand{\lp}{\langle}
 \newcommand{\rp}{\rangle}

\newtheorem{thm}{Theorem}[section]
\newcommand{\bt}{\begin{thm}}
\newcommand{\et}{\end{thm}}

\newtheorem{conj}[thm]{Conjecture}

\newtheorem{cor}[thm]{Corollary}   
\newcommand{\bc}{\begin{cor}}
\newcommand{\ec}{\end{cor}}

\newtheorem{lem}[thm]{Lemma}   
\newcommand{\bl}{\begin{lem}}
\newcommand{\el}{\end{lem}}

\newtheorem{prop}[thm]{Proposition}
\newcommand{\bp}{\begin{prop}}
\newcommand{\ep}{\end{prop}}

\newtheorem{defn}[thm]{Definition}
\newcommand{\bd}{\begin{defn}}    
\newcommand{\ed}{\end{defn}}

\newtheorem{rmrk}[thm]{Remark}   

\newcommand{\br}{\begin{rmrk}}
\newcommand{\er}{\end{rmrk}}

\newcommand{\thmref}[1]{Theorem~\ref{#1}}
\newcommand{\secref}[1]{Section~\ref{#1}}
\newcommand{\lemref}[1]{Lemma~\ref{#1}}
\newcommand{\defref}[1]{Definition~\ref{#1}}

\newcommand{\R}{\mathbb{R}}

\newcommand{\vol}{{\rm vol}}
\newcommand{\Sc}{{\rm Sc}}
\newcommand{\Ric}{{\operatorname{Ric}}}

\newcommand{\Hess}{{\rm Hess}}

\begin{document}

\title[Singular metrics with nonnegative scalar curvature and RCD]{Singular metrics with nonnegative scalar curvature \\
and RCD}

\author{Xianzhe Dai}
\address{
Department of Mathematics,
University of California, Santa Barbara
CA93106, USA}
\email{dai@math.ucsb.edu}
\thanks{XD was partially supported by Simons Foundation}

\author{Changliang Wang}
\address{
School of Mathematical Sciences and Institute for Advanced Study, Key Laboratory of Intelligent Computing and Applications(Ministry of Education), Tongji University, Shanghai 200092, China}
\email{wangchl@tongji.edu.cn}
\thanks{CW was partially supported by the Fundamental Research Funds for the Central Universities and Shanghai Pilot Program for Basic Research.}

\author{Lihe Wang}
\address{
Department of Mathematics, The University of Iowa,  Iowa City, IA 52242-1419 USA
}
\email{lihe-wang@uiowa.edu}

\author{Guofang Wei}
\address{
Department of Mathematics,
University of California, Santa Barbara
CA93106, USA
}
\thanks{GW was partially supported by NSF DMS grant  2403557}
\email{wei@math.ucsb.edu}

\date{}

\keywords{}

\begin{abstract}
We show that a uniformly Euclidean metric with isolated singularity on $M^n = T^n \# M_0$, where $4\leq n\leq 7$ or $n\geq 4$, $M_0$ spin, and nonnegative scalar curvature on the smooth part is Ricci flat and extends smoothly over the singularity. This confirms Schoen's Conjecture in these cases. The key to the proof is to show that the space has nonnegative synthetic Ricci curvature, i.e.,  an $RCD(0, n)$ space. Our result also holds when the singular set consists of a finite union of submanifolds (of possibly different dimensions) intersecting transversally under additional assumption on the co-dimension and the location of the singular set.
\end{abstract}

\maketitle

\tableofcontents

\section{Introduction}
The existence problem of Riemannian metrics with positive scalar curvature is a fundamental topic in Riemannian geometry. In the smooth setting, it has been well studied and many important results have been established.  Kazdan-Warner \cite{KW-JDG-75} and Schoen \cite{Schoen-89} proved that on a closed manifold, there exists a smooth metric with positive scalar curvature iff its Yamabe constant (aka $\sigma$-constant or Schoen constant) is positive. Moreover, their rigidity result says that on a closed manifold with a nonpositive Yamabe constant, any metric with nonnegative scalar curvature must be Ricci flat. 
Recall that, for a closed manifold $M$, its Yamabe constant $\sigma(M)$ is defined as
\begin{equation*}
\sigma(M) := \sup \{Y(M, [g_0]) \mid [g_0] \text{ is a conformal class of smooth metrics on } M \},
\end{equation*}
where 
\begin{equation*}
Y(M, [g_0]) := \inf \left\{\int_{M} \Sc_{g} d\vol_{g} \mid g \in [g_0], \ \ {\rm Vol}_g(M) =1\right\}
\end{equation*}
and $\Sc_g$ denotes the scalar curvature of a smooth metric $g$. The Yamabe constant $\sigma(M)$ is a diffeomorphism invariant of $M$. 
For example, for a torus $T^n$, $\sigma(T^n) =0$ as a consequence of the Geroch conjecture established by Schoen-Yau \cite{SY-MM-79} and Gromov-Lawson \cite{GL-Annals-80}. 

For simply connected manifolds of dimension $\ge 5$, the problem of what closed manifolds should admit metrics of
positive scalar curvature is
fully understood in terms of the second Stiefel–Whitney class (which determins whether the manifold is spin or not) and $\alpha$-invariant \cite{Gromov-Lawson1980-simply, Stolz1992}.

Motivated by general relativity, the Geroch conjecture says that there is no smooth metric of positive scalar curvature on $T^n$. This was first proved by Schoen-Yau \cite{SY-MM-79} in dimensions $ n \leq 7 $ using the minimal surface method, and later by Gromov-Lawson \cite{GL-Annals-80, GL-IHESP-83} in arbitrary dimensions using the Dirac operator method. More generally, Gromov-Lawson\cite{GL-IHESP-83} proved that a $\Lambda^2$-enlargeable manifold does not admit any complete metric with positive scalar curvature. Moreover, one has the rigidity result which says that the metric with nonnegative scalar curvature on such manifolds must be flat. In particular, for any closed spin manifold $M^n$, $T^n \# M$ is $\Lambda^2$-enlargeable, and as a result, $\sigma(T^n \# M) \leq 0$.  Recently, these results have been generalized to allow $M$ to be non-compact by Chodosh-Li \cite{Chodosh-Li-Annals} for dimensions $\leq 7$, Wang-Zhang \cite{Wang-Zhang-22} for spin manifolds of arbitrary dimensions, and also Lesourd-Unger-Yau \cite{LUY-JDG} in dimension three,

Inspired by the investigation of weak notions of nonnegative scalar curvature, e.g., Gromov's polyhedral comparison theory \cite{Gromov14}, and the positive mass theorem for singular metrics, it is natural to study the forementioned Kazdan-Warner result and Geroch conjecture for singular metrics. A natural class of such singular metrics is that of uniformly Euclidean metrics. 
Following Li-Mantoulidis \cite{Li-Mantoulidis}, 
we say a measurable symmetric $2$-tensor $g$ on a closed manifold $M$ to be a {\em uniformly Euclidean metric} (also called $L^{\infty}$ metric), if 
\begin{equation}\label{eqn-uniformly-Euclidean-metric}
\Lambda^{-1} \bar{g} \leq g \leq \Lambda \bar{g}
\end{equation}
holds a.e. on $M$ for some smooth metric $\bar{g}$ on $M$ and constant $\Lambda > 0$. A point is a smooth point for $g$ if there is an open neighborhood in which $g$ is smooth. Otherwise, we will call it a singular point. The singular set is the collection of singular point; it is closed by definition as its complement, called the smooth part of $g$, is an open set.
We say $g$ has positive (nonnegative resp.) scalar curvature if its scalar curvature on the smooth part is positive (nonnegative resp.).

There are two basic problems in this direction. First of all, just like the smooth case, one has the existence/non-existence problem of singular uniformly Euclidean metrics with positive scalar curvature on a closed manifold with nonpositive Yamabe constant. Secondly, related to the rigidity part, a new challenge here is the smooth extension/removable singularity problem of uniformly Euclidean metrics with nonnegative scalar curvature. Note that once the metric can be smoothly extended over the singular set, the rigidity result directly follows from the Kazdan-Warner result. For uniformly Euclidean metrics with conical singularities along a submanifold of codimension 2, Li-Mantoulidis \cite{Li-Mantoulidis} obtained an affirmative answer to both problems, provided the angle of the cone $\leq 2\pi$, while they also provided counterexample when the angle of the cone $> 2\pi$. In the case of singular set having codimension $\geq 3$, Schoen has the following conjecture.
\begin{conj}[ Conjecture 1.5 in \cite{Li-Mantoulidis}]\label{conj-Schoen}
{\rm
Suppose $M$ is a closed smooth manifold whose Yamabe constant is nonpositive, and $g$ is a uniformly Euclidean metric on $M$ that is smooth away from a closed embedded submanifold $S \subset M$ with codimension $ \geq 3$. Then $g$ extends smoothly to $M$, and hence is Ricci flat. In other words, 
\begin{eqnarray*}
& & \Sc_g \geq 0 \ \  \text{on} \ \ M \setminus S,  \ \ \ \sigma(M) \leq 0, \ \text{and}\ co\dim(S) \geq 3  \\
& \Longrightarrow & g \ \ \text{extends smoothly to } \ \ M \ \ \text{and} \ \ \Ric_g \equiv 0.
\end{eqnarray*}
}
\end{conj}

Li-Mantoulidis \cite{Li-Mantoulidis} proved this conjecture for $3$-dimensional $M$. They used a removal singularity result of Smith-Yang \cite{Smith-Yang} for $3$-dimensional Einstein metrics. On the other hand, very recently,
Cecchini-Frenck-Zeidler \cite{Cecchini2024} constructed counterexamples to Conjecture \ref{conj-Schoen} on simply connected closed manifold of dimension $\geq 8$. In contrast, the conjecture could still be true on manifolds with certain topology. Dai-Sun-Wang \cite{DSW-PSC-conical} proved Conjecture \ref{conj-Schoen} for metrics with isolated conical singularities on $T^n \# M_0$, including a smooth extension and rigidity result, provided that $3 \leq n \leq 7$ or $M_0$ is spin. On $T^n \# M_0$, the non-existence problem of general uniformly Euclidean metrics with the singular set having co-dimension $\geq \frac{n}{2} +1$ can be answered affirmatively in a similar way. Kazaras \cite{Kazaras-MAAN-19} proved that, on a four-dimensional enlargeable closed smooth manifold, a uniformly Euclidean
metric g that is smooth away from a smoothly embedded finite 1-complex and has nonnegative
scalar curvature on the smooth part must be Ricci flat on the smooth part. 
Cecchini-Frenck-Zeidler \cite{Cecchini2024} proved that a uniformly Euclidean metric on a closed spin manifold with the fundamental group satisfying some condition on Rosenberg index must be Ricci flat on the smooth part, provided that the singular set consisting of finitely many points and the scalar curvature is nonnegative on the smooth part.  Wang-Xie \cite{WX2024} also proved a similar result on manifolds admitting a Lipschitz map to $T^n$ with nonzero degree and allowing the singular set having codim $\geq \frac{n}{2}+1$. Shi-Wang-Wu-Zhu \cite{SWWZ-MAAN-21} defined open Schoen-Yau-Schick manifolds and proved that there is no complete metric with positive scalar curvature on such manifolds and consequently proved that a uniformly Euclidean metric with nonnegative scalar curvature and a singular set of codimension $\geq \frac{n}{2}+1$ on a closed Schoen-Yau-Schick manifold must be Ricci flat on the smooth part. Kazaras \cite{Kazaras-MAAN-19}, Cecchini-Frenck-Zeidler \cite{Cecchini2024}, Wang-Xie \cite{WX2024} and Shi-Wang-Wu-Zhu \cite{SWWZ2024} obtained the Ricci flatness of the metrics on the smooth part, but did not establish a smooth extension result. 

The main result of this paper is a partial affirmative solution of Conjecture \ref{conj-Schoen} for manifolds which are connected sums with the torus.

\begin{thm}\label{thm-smooth-extension}
	Let $M^n = T^n \# M_0$, $n \geq 4$, be a connected closed manifold. Assume either $4 \leq n\leq 7$ or $M_0$ is spin. Suppose \\
1). $g$ is a uniformly Euclidean metric on $M$ that is smooth away from a closed, embedded submanifold $S$ with codimension
		$\geq \max (4, \frac{n}{2}+1)$. \\
2). There is a $(n-1)$-sphere dividing $M=T^n\# M_0$ into the $T^n$-part and $M_0$-part, such that each component of $S$ lies either in the $T^n$-part or $M_0$-part. Further, if a component of $S$ lies in the torus part, we assume that it is contained in a topological ball. \\
3). The scalar curvature $\Sc_g \geq 0$ on $M\setminus S$.\\ 
	 Then $g$ extends smoothly to $M$ (possibly with a change of differential structure)  and $(M^n, g)$ is isometric to a flat torus.
\end{thm}

In particular, this confirms Schoen's Conjecture for isolated singularity when $M^n = T^n \# M_0$, $4 \leq n \leq 7$ or $n\geq 4$ and $M_0$ is spin.

\begin{rmrk}
{\rm
 As pointed out in \cite{Cecchini2024}, one may need to change differential structure when extending the metric. This is also the case for other well-known removable singularity results.  Our result also holds for singular set which consists of a finite union of submanifolds (of possibly different dimensions) intersecting transversally. 
 }
\end{rmrk}

\begin{rmrk}
{\rm
One may think of this result as a removable singularity theorem. Unlike the well-known removable singularity results for Einstein metrics or Yang-Mills instantons, our result is not local but rely on the global topological structure.
}
\end{rmrk}

A novel and crucial ingredient in the proof of Theorem \ref{thm-smooth-extension} is the following result, which is of independent interest. It says that a uniformly Euclidean metric on a closed manifold smooth away from a closed singular set of codimension $4$ or higher and with  Ricci curvature bounded below by $K$ on the smooth part is {\rm RCD(K, n)}, i.e., its Ricci curvature is bounded below by $K$ in the weak sense.

\begin{thm}\label{thm-RCD}
Let $M^n$ be a connected closed manifold of dimension $n\geq 4$, and $g$ a uniformly Euclidean metric on $M$ that is smooth away from a closed set   
$S$ with codimension $ 
\geq 4$. If, for some constant $K$, the Ricci curvature $\Ric_g \geq K$ on the smooth part, then $(M^n, g)$ is a {\rm RCD(K, n)} space.
\end{thm}

As an interesting application, we obtain the following removable singularity result. In contrast to the local nature of the previous results of this type, our result is global, depending crucially on the global topological assumption.

\begin{thm}\label{thm-removable}
Let $M^n$ be a connected closed manifold of dimension $n\geq 4$ with infinite fundamental group.  If $g$ is a uniformly Euclidean metric on $M$ with isolated singularity and nonnegative Ricci curvature on the smooth part, then $(M^n, g)$ extends smoothly to $M$ and thus $\Ric_g \geq 0$. Moreover, a finite cover $\hat{M}$ of $M$ is diffeomorphic to $T^k \times X$, for some $1\leq k \leq n$ and $X$ a closed simply connected manifold. Furthermore, there is a locally isometrically trivial fibration $\hat{M} \longrightarrow T^k$ with fiber $X$.
\end{thm}

The main technical component of the proof of Theorem \ref{thm-RCD} is a gradient estimate for eigenfunctions. It consists of two steps. First, using De Giorgi-Nash-Moser estimate and Cheng-Yau gradient estimate, we obtain an initial gradient estimate for eigenfunctions, which blows up at the singular part of $g$. Nevertheless, this enables us to establish an $L^2$ bound on the Hessian of the eigenfunctions, using the uniform Euclidean and codimension assumptions. Then, by Bochner formula and Kato's inequality, a differential inequality for the norm of the gradient of an eigenfunction is shown to hold over the whole manifold in the weak sense, thanks to the $L^2$ Hessian bound. The final gradient estimate then follows by the Moser iteration. This is similar in spirit to the strategy used in \cite{Mondello-thesis,BKMR-AIF-21} for stratified spaces.

This paper is organized as follows. In \secref{sect-removable-singularity-via-RCD}, we prove \thmref{thm-RCD} and then apply it to prove Theorems \ref{thm-smooth-extension} and \ref{thm-removable}. More precisely, in Subsection \ref{subsect-RCD}, we recall a definition of RCD space and describe a metric measure structure induced by a uniformly Euclidean metric as in Conjecture \ref{conj-Schoen}. We also prove that such metric measure spaces satisfy the infinitesimal Hilbertian condition and Sobolev-to-Lipschitz property. In Subsection \ref{subsect-gradient-estimate}, we derive a gradient estimate for the eigenfunctions of the Laplace operator, which plays an important role in the proof of the Bakry-\'Emery (BE, for short) condition for uniformly Euclidean metrics.  
Based on these, in Subsection \ref{subsect-proof-main-results}, we prove \thmref{thm-RCD}, which follows from 
\cite[Theorem 3.7]{Honda-AGMS-18}; see also Bertrand-Ketterer-Mondello-Richard's work on stratified spaces \cite{BKMR-AIF-21}, as well as related work of Jiang-Sheng-Zhang \cite{Jiang-Sheng-Zhang2022} and Szekelyhidi \cite{Szekelyhidi2024}. By applying Theorem \ref{thm-RCD}, we then prove Theorems \ref{thm-smooth-extension} and \ref{thm-removable}, and give several interesting examples as their applications. These examples illustrate how naturally our approach of using metric measure spaces deals with the possible change of differential structures, as well as situations outside Theorem \ref{thm-RCD} our results still apply. Finally, in Appendix \ref{sect-appendix}, we provide a detailed proof of \thmref{gradi}.

{\em Acknowledgements:} This material is based upon work supported by the National Science Foundation under Grant No. DMS-1928930, while XD and GW were in residence at the Simons Laufer Mathematical Sciences Institute
(formerly MSRI) in Berkeley, California, during the Fall semester of 2024. XD and GW thank  SLMath for providing a stimulating environment and acknowledge useful discussions with Rick Schoen, Vincent Guedj, Xinyu Zhu, Conghan Dong.  XD and CW would like to thank Jian Wang and Simone Cecchini for valuable conversations. Finally we would like to thanks Gabor Szekelyhidi for very fruitful communication.

\newpage

\section{Removable singularity via RCD}\label{sect-removable-singularity-via-RCD}

\subsection{RCD spaces and uniformly Euclidean metric}\label{subsect-RCD}
We recall a definition of the RCD$(K, N)$ condition for metric measure spaces, see \cite{Ambrosio-Mondino-Savare2019, Erbar-Kuwada-Sturm2015, Gigli-MemAMS-2015}.
\begin{defn}\label{defn-RCD}
A metric measure space $(X, d, m)$ is said to be an ${\rm RCD}(K, N)$ space for some $K \in \R$ and $N \in [1, \infty]$, if
\begin{enumerate}[$(1)$]
\item there exists $x_0 \in X$ and a constant $C>0$ such that $\int_{X} e^{-C d(x, x_0)} dm < +\infty$, 
\item {\rm (Infinitesimally Hilbertian condition)} the Sobolev space $W^{1, 2}(X, d, m)$ is a Hilbert space,
\item {\rm (Sobolev-to-Lipschitz property)} for any $f \in W^{1, 2}(X, d, m)$ satisfying $|\nabla f| \in L^{\infty}(M)$, it has a Lipschitz representative $\tilde{f}$ with ${\rm Lip}(\tilde{f}) \leq \|\nabla f\|_{L^\infty(X)}$,
\item {\rm (BE(K, N) condition)} for any $f \in D(\Delta)$, 
and any test function $\varphi \in L^{\infty}(X) \cap D(\Delta)$ with $\varphi \geq 0$ and $\Delta \varphi \in L^\infty(X)$, the weak Bochner inequality 
             \begin{equation}\label{eqn-BE-inequality}
             \frac{1}{2} \int_{X} |\nabla f|^2 \Delta \varphi dm \geq \int_{X} \varphi \left(  \tfrac{1}{N}(\Delta f)^2 + \langle \nabla f, \nabla \Delta f \rangle + K |\nabla f|^2 \right) dm             
             \end{equation}
             holds.
\end{enumerate}
\end{defn}

In this paper, we focus on a closed connected smooth manifold $M^n$ with a uniformly Euclidean metric $g$ as in Conjecture \ref{conj-Schoen}, that is, $g$ is smooth away from a closed singular set $S$ with codimension $\geq 3$. We will denote by $\Omega := M \setminus S$ the smooth part of the metric and by $k := \dim(S)$. 

We describe a metric measure structure on such $(M^n, g)$ as follows. 
We say that a piecewise $C^1$ curve $\gamma: [a, b] \to M$ is {\em admissible} if its image $\gamma([a, b])$ intersects with $S$ at only finitely many points. For such a curve, its length is defined as $L_g(\gamma) = \int^b_a |\gamma^\prime(t)|_g dt$. Then, the distance between any $x, y \in M$ is defined as
\begin{equation*}
    d_g(x, y) := \inf \{ L_g(\gamma) \mid \gamma: [a, b] \to M \ \ \text{is admissible curve with} \ \ \gamma(a)=x, \gamma(b)=y \}.
\end{equation*}

The volume measure $d\mu_g$ is defined in the usual way. A direct consequence of the uniformly Euclidean condition is
\begin{equation}
    \Lambda^{-\frac{n}{2}} d\mu_{\bar{g}} \leq d\mu_g \leq \Lambda^{\frac{n}{2}} d\mu_{\bar{g}},
\end{equation}
where $d\mu_{\bar{g}}$ is the volume measure of the smooth metric $\bar{g}$ on $M$. As a result, on a sufficiently small tubular neighborhood $T_\epsilon(S)$ of the singular set $S$, we have that locally,
\begin{equation}\label{eqn-volume-measure-estimate}
    d\mu_g \sim r^{n-k-1} dr d\theta^{n-k-1} d\mu_{\bar{g}_S},
\end{equation}
where $r = r(x) := d(x, S)$, $d\theta^{n-k-1}$ is the volume measure on $\mathbb{S}^{n-k-1}$, and $d\mu_{\bar{g}_S}$ is the volume measure on $S$ with the restricted metric $\bar{g}_S$. Here $\sim$ means bounded above and below by the right hand side (not precise asymptotic).  By applying this volume measure estimate, we prove the following zero capacity property of $S$.

\begin{lem}[Zero capacity property]\label{lem-zero-capacity-property}
Let $M^n$, $n \geq 3$, be a closed manifold, and $g$ a uniformly Euclidean metric on $M$ with singular set a closed embedded submanifold of codimension $\geq 3$. There exists a family of functions $\psi_\delta \in C^\infty_0(\Omega)$, $\delta \in (0, \delta_0)$ for some $\delta_0>0$, such that the following two properties hold:
\begin{enumerate}[$(1)$]
\item for any compact subset $A \subset \Omega := M \setminus S$, $\psi_\delta |_{A} \equiv 1$ hold for all sufficiently small $\delta$,
\item $0 \leq \psi_\delta \leq 1$ for all $\delta$ and $|\nabla \psi_\delta| \leq \frac{C}{\delta}$. Moreover,
         \begin{equation}
         \int_{\Omega} |\nabla \psi_\delta|^2 d\mu_g \to 0, \quad \text{as} \ \ \delta \to 0.
         \end{equation}
\end{enumerate}
\end{lem}
\begin{proof}
Take a smooth cut-off function $\eta: [0, \infty) \to [0, 1]$ such that 
      \begin{equation}
      \eta = \begin{cases}
                 0, & \text{in} \ \ [0, \tfrac{1}{2}], \\
                 1, & \text{in} \ \ [1, \infty),
                \end{cases}
      \end{equation}
 and $|\eta^\prime| \leq C$ on $[0, \infty)$ for some constant $C$. Let $r(x) := d(x, S)$, and $\psi_\delta (x) := \eta \left( \tfrac{r}{\delta} \right)$. Clearly, $|\nabla \psi_\delta| \leq \frac{C}{\delta}$ and ${\rm supp}(\nabla \psi_\delta) \subset A_{\frac{\delta}{2}, \delta} (p) :=\{x \in M \mid \tfrac{\delta}{2} \leq  d(x, S) \leq \delta \}$. By combining with the estimate of the volume measure in \eqref{eqn-volume-measure-estimate}, this implies
     \begin{equation}
     \int_{\Omega} |\nabla \psi_\delta|^2 d\mu_g 
     \leq \frac{C^2}{\delta^2} \int_{A_{\frac{\delta}{2}, \delta}(p)} d\mu_g 
     \leq \frac{C^\prime}{\delta^2} \int^{\delta}_{\frac{\delta}{2}} r^{n-k-1} dr  
     \leq {C^{\prime\prime}} \delta^{n-k-2} \to 0,
     \end{equation}
     as $\delta \to 0$, since $n -k \geq 3$. This completes  the proof.
\end{proof}

By using \lemref{lem-zero-capacity-property}, one can show that the Sobolev space $H^{1, 2}(M)$ defined as the $H^{1, 2}$-completion of $C^\infty_0 (M \setminus S)$ coincides with the Sobolev space $W^{1 ,2}(M, d_g, d\mu_g)$ on the metric measure space $(M, d_g, d\mu_g)$, see Remark 3.2 and Proposition 3.3 in \cite{Honda-AGMS-18}, for details. In particular, $(M, d_g, d\mu_g)$ satisfies the infinitesimally Hilbertian condition in \defref{defn-RCD}.

For $f \in W^{1, 2}(M)$, if there exists a function $u \in W^{1, 2}(M)$ such that 
\begin{equation*}
    \int_M uv d\mu_g = - \int_M \langle df, dv \rangle_g d\mu_g
\end{equation*}
holds for any test function $v \in C^\infty_0(M \setminus S)$, then $\Delta f := u$ is said to be the Laplace of $f$, and the domain of the Laplace operator is 
\begin{equation*}
    D(\Delta):= \{f \in W^{1, 2}(M) \mid \Delta f \in W^{1, 2}(M)\}.
\end{equation*}

For proving the Sobolev-to-Lipschitz property for uniformly Euclidean metrics, we need the following result about the length of admissible curves.
\begin{lem}\label{lem-admissible-curve}
Let $M^n$, $n \geq 3$, be a closed connected manifold, and $g$ a uniformly Euclidean metric on $M$ as in Conjecture \ref{conj-Schoen}. Let $\gamma: [a, b] \to M$ be an admissible curve on $M$. For any $\epsilon >0$, there exists an admissible curve $\gamma_\epsilon$ with the same endpoints as $\gamma$ such that the image of $\gamma_\epsilon$ is contained in the regular part $M \setminus S$ probably except the endpoints, and $L_g (\gamma_\epsilon) \leq L_g(\gamma) + \epsilon$.
\end{lem}
\begin{proof}
Because $\gamma$ intersects with $S$ at only finitely many points and the length is additive, without loss of generality, we assume that $\gamma$ intersects with $S$ only at an interior point $p$. Then we only need to modify the curve $\gamma$ near $p$. For that, we take a small ball $B_\delta(p)$ centered at $p$ with radius $\delta$ with respect to the smooth metric $\bar{g}$. Let $\gamma(t_1)$ be the first point on $\gamma$ that intersects $\partial B_\delta(p)$, and $\gamma(t_2)$ be the last such point, where $[t_1, t_2] \subset (a, b)$. For sufficiently small $\delta>0$, we can take a smooth curve $c$ in $B_\delta(p)$ that connects $\gamma(t_1)$ and $\gamma(t_2)$, has length $< \frac{\epsilon}{\sqrt{\Lambda}}$ with respect to $\bar{g}$, and does not intersect with $S$. Now, define $\gamma_\epsilon$ to be the concatenation of $\gamma|_{[a, t_0]}$, $c$ and $\gamma|_{[t_1, b]}$. Then the image of $\gamma_\epsilon$ intersects with $S$ at most at the endpoints. Moreover, $L_g(\gamma_\epsilon) \leq L_g(\gamma) + L_g(c) \leq L_g(\gamma) + \epsilon$.
\end{proof}

\begin{lem}\label{lem-Sobolev-to-Lipschitz}
Let $M^n$, $n \geq 3$, be a closed connected manifold, and $g$ a uniformly Euclidean metric on $M$ as in Conjecture \ref{conj-Schoen}. Any $f \in W^{1, 2}(M)$ with gradient $\nabla f \in L^2(M) \cap L^\infty(M)$ has a Lipschitz representative $\tilde{f}$ whose Lipschitz constant $\leq \|\nabla f\|_{L^\infty(M)}$.
\end{lem}
\begin{proof}
First note that the smooth part $\Omega:= M \setminus S$ is path connected, since the singular set $S$ has codimension $\geq 3$. Take any two points $x, y \in \Omega$. By \lemref{lem-admissible-curve}, for any $\epsilon >0$, there exists a piecewise smooth curve $\gamma_\epsilon \subset \Omega$ from $x$ to $y$ such that the length $L_g(\gamma_\epsilon)$ of $\gamma$ with respect to $g$
\begin{equation}
    L_g(\gamma) \leq (1+\epsilon) d_g(x, y).
\end{equation}
Then by locally Lipschitz property of $f$ in $\Omega$ with Lipschitz constant $\leq \|\nabla f\|_{L^\infty(M)}$, and applying the fundamental theorem of calculus to $u \circ \gamma$, we have
\begin{equation*}
    |f(x) - f(y)| \leq (1+\epsilon) \|\nabla f\|_{L^\infty(M)} d_g(x, y).
\end{equation*}
Thus, by letting $\epsilon \to 0$, we have that $u$ is uniformly Lipschitz on $\Omega$ with Lipschitz constant $\leq  \|\nabla f\|_{L^\infty(M)}$. Finally, because $\Omega$ is dense in $M$, $u$ has a unique Lipschitz continuous extension $\tilde{f}$ on $M$, which equals $f$ almost everywhere on $M$.
\end{proof}

The uniformly Euclidean property directly implies the following $L^2$-strong compactness property.
\begin{lem}[$L^2$-strong compactness]\label{lem-L2-strong-compactness}
Let $M^n$, $n \geq 3$, be a closed manifold, and $g$ a uniformly Euclidean metric on $M$. 
The inclusion map $i: W^{1, 2}(M, g) \hookrightarrow L^2(M, g)$ is compact.
\end{lem}

\begin{proof}
    By the definition, the Sobolev space $W^{1,2}(M, g)$ and $L^2(M, g)$ depend only on the quasi-isometry class of $g$ and hence is equivalent to $W^{1,2}(M, \bar{g})$ for a smooth metric $\bar{g}$. Clearly $W^{1,2}(M, \bar{g})$ embeds compactly in $L^2(M, \bar{g})$ on a compact manifold.
\end{proof}
Using the $L^2$-strong compactness in Lemma \ref{lem-L2-strong-compactness} and the standard argument in functional analysis, one obtains that the spectrum of the Laplace operator $\Delta_g$ consists of discrete eigenvalues with finite multiplicities, and normalized eigenfunctions form an orthonormal basis of $L^2(M, g)$.




\subsection{
Gradient estimate of eigenfunctions }\label{subsect-gradient-estimate}

For proving the BE$(K, n)$ condition for uniformly Euclidean metrics, as in Theorem 3.7 in \cite{Honda-AGMS-18}, we expand a Sobolev function in terms of eigenfunctions of $\Delta_g$, and a gradient estimate of eigenfunctions plays a crucial role in the proof. In this subsection, we derive some estimates for gradient and Hessian of the eigenfunctions.

Let $M^n$ be a closed manifold, and $g$ a uniformly Euclidean metric on $M$ that is smooth away from a closed embedded submanifold $S$. 
As $M$ is a closed manifold and $g$ is uniformly Euclidean, for any $q\in \Omega$, $0< 2r < d(q, S)$, the gradient estimate of Cheng-Yau  \cite{Cheng-Yau1975} still applies on $B_r(q)$, provided one has a Ricci curvature lower bound on $B_{2r}(q)$. 
The following combines the versions in \cite{Cheeger2001} and \cite{Li-Wang2002}.  

\begin{thm}  \label{gradi}
	Let $(M^n, g)$ be as above and $\Omega := M \setminus S$ the regular part of $g$. If $u$ is a positive function defined
on the closed ball $B(q, 2r) \subset \Omega$ that satisfies $\Delta u = K(u)$ and moreover $\Ric_{g} \ge
	-(n-1)H^2\ (H \ge 0)$ on $B(q, 2r)$. Then on  $B(q,r)$,  \ban \label{gradient} 
   \lefteqn{ \frac{|\nabla u(x)| }{u} \le 
    \max \left\{  \sup_{B(q, 2r)} \left(\tfrac{|K(u)|}{u}\right)^{\frac 12},  \right. } \\ & & 
   \left.    (n-1) H + C_1(n) r^{-1} + C_2(n,H) r^{-\frac 12}  + \sup_{B(q, 2r)} \left[(n-1) \left| K'- \tfrac{n+1}{n-1} \tfrac{K}{u}\right|\right]^{\frac 12} \right\}. 
    \ean
\end{thm}

Since this version is not in the literature, for completeness we provide a proof in the appendix. 

Here, as an application, we deduce the following crucial estimate.

\begin{lem}\label{lem-eigenfunction-estimate}
 Let $(M^n, g)$ be as above.

\begin{enumerate}[$(1)$]
\item  There exists $0< \alpha=\alpha(n, \Lambda) <1$, such that the $W^{1,2}$-eigenfunctions $u$ of the Laplace operator $\Delta_g$ satisfy $u \in C^{\alpha}(M)$. Moreover, if $\lambda$ is the eigenvalue, 
            \begin{equation}\label{DeGiorgi}
            [u]_{C^{\alpha}(M)} \leq C(n, \Lambda, \lambda, \bar{g}) \|u\|_{L^{\infty}(M)}\leq C\|u\|_{L^2(M)}.
            \end{equation}
\item   
Assume that $Ric_g \geq K$ on $\Omega$ for some constant $K$. 
 For all $x\in \Omega$ with $r(x)=d(x, S)\leq 1$, we have the gradient estimate
              \begin{equation}
              |\nabla u|(x) \leq \frac{C(n, \Lambda, \lambda, \bar{g}, K)}{r(x)^{1-\alpha}}\|u\|_{L^2(M)},
              \end{equation}
              for the  same $\alpha $  in (\ref{DeGiorgi}).
\end{enumerate}
\end{lem}

\begin{proof} Since $M$ is uniformly Euclidean
    the first statement (1) is the De Giorgi-Nash-Moser estimate, see Chapter 8 in \cite{GT2001}, \cite{LSW63} for example.  

     For (2), in order to apply Theorem \ref{gradi} we do a shift to get a positive solution. Namely  fix any $y\in \Omega$ and let $r(y)=2s>0$.  Then $v=u-m+s\geq s$ with $m=\inf\{u: B_{2s}(y)\}$ 
    is a positive solution of
   \begin{equation}\label{x-p}
    -\Delta v=\lambda v+\lambda (m-s)
       \end{equation}
   in the ball $B_{2s}(y) \subset \Omega$. 
  Since $\Ric \ge K$ on $\Omega$, by Theorem \ref{gradi}, we have  on $B_s(y)$
    \begin{align*}
    |\nabla u| & = |\nabla v| \le v \left((n-1)K + C(n) s^{-1} + C(n, K)s^{-\frac 12} + \left[(n-1) (2\lambda + \lambda \frac{|m|+s}{s})\right]^{1/2}\right) \\
    & = v \, s^{-1} \left( (n-1)Ks +C(n) + C(n, 
    K) \sqrt{s} +  \sqrt{ (n-1)s \lambda} \, (3s + |m|)^{1/2} \right). 
    \end{align*}

 By definition, there is a $q\in \overline{B_{2s}(y)}$ such that 
 $$|v(y)|\le |u(y)-u(q)|+s.$$ 
 Now by (\ref{DeGiorgi}), 
 $$ |u(y)-u(q)| \leq C s^\alpha \, \|u\|_{L^2(M)}. $$
 Hence, for $s \le 1$, 
 \[ 
 |\nabla u (y)|\le C(n, \Lambda, \lambda,\bar{g}, K)s^{\alpha} s^{-1} \, \|u\|_{L^2(M)}.
 \]
 Our estimate follows since $y\in \Omega$ is arbitrary.
\end{proof}

By applying the gradient estimate in \lemref{lem-eigenfunction-estimate}, we prove the $L^2$-integrability of the Hessian of eigenfunctions as in the following lemma.

\begin{lem}\label{lem-eigenfunction-Hess-L2}
Let $M^n$ 
be a closed manifold, and $g$ a uniformly Euclidean metric on $M$ that is smooth away from a closed embedded submanifold $S$ with codimension $n-k\geq 4$, and $Ric_g \geq K$ on $\Omega$. 
For any $W^{1, 2}$-eigenfunction $u$ of $\Delta_g$, $|\Hess\, u| \in L^{2}(M)$.
\end{lem}
\begin{proof}
    Let $u$ be a $W^{1, 2}$-eigenfunction of $\Delta$ with eigenvalue $\lambda$, i.e. $-\Delta u = \lambda u$ in the weak sense. By standard elliptic regularity theory, $u \in C^\infty(\Omega)$.
    Bochner formula and $\Ric_g \geq K$ on $\Omega$ imply
    \begin{equation}\label{eqn-Bochner-inequality-regular-part-eigenfunction}
        \frac{1}{2} \Delta |\nabla u|^2 
        \geq |\Hess\, u|^2 + (K - \lambda) |\nabla u|^2 \ \ \text{on} \ \ \Omega.
    \end{equation} 
    Multiplying the inequality (\ref{eqn-Bochner-inequality-regular-part-eigenfunction}) by $\psi^2_\delta$ in Lemma \ref{lem-zero-capacity-property}, and then integrating it over $M$ gives
\begin{equation}\label{eqn-integral-Bochner-inequality-eigenfunction}
\frac{1}{2} \int_{M} |\nabla u|^2 \Delta \psi^2_\delta d\mu_g  
\geq 
\int_{M} \psi^2_\delta \left( |\Hess\, u|^2 + (K - \lambda) |\nabla u|^2 \right) d\mu_g.
\end{equation}
For the left hand side, 
\begin{eqnarray*}
\int_{M} |\nabla u|^2 \Delta \psi^2_\delta d\mu_g 
& = & - 4 \int_{M} |\nabla u| \psi_\delta \langle \nabla |\nabla u|, \nabla \psi_\delta \rangle d\mu_g \\
& \leq & \int_{M} \psi^2_\delta |\Hess\, u|^2 d\mu_g + 4 \int_{M} |\nabla u|^2 |\nabla \psi_\delta|^2 d\mu_g.
\end{eqnarray*}
Plugging this into (\ref{eqn-integral-Bochner-inequality-eigenfunction}) and rearranging the equation give
\begin{equation*}
\frac{1}{2} \int_{M} \psi^2_\delta |\Hess\, u|^2 d\mu_g
 \leq  \int_{M}2 |\nabla u|^2 |\nabla \psi_\delta|^2 d\mu_g  - (K - \lambda)) \int_M \psi^2_\delta |\nabla u|^2 d\mu_g.
 \end{equation*}
 Then by applying estimates in Lemmas \ref{lem-zero-capacity-property} and \ref{lem-eigenfunction-estimate} and equation \eqref{eqn-volume-measure-estimate} for the integrand in the first term on the right hand side, we obtain
 \begin{eqnarray*}
 \frac{1}{2} \int_{M} \psi^2_\delta |\Hess\, u|^2 d\mu_g
 & \leq & C \int^{\delta}_{\frac{\delta}{2}} r^{2\alpha + n -k-5}dr - (K - \lambda) \int_M \psi^2_\delta |\nabla u|^2 d\mu_g  \\
& \leq & \tilde{C} \delta^{2\alpha -4 +n -k} - (K - \lambda) \int_M \psi^2_\delta |\nabla u|^2 d\mu_g.
\end{eqnarray*}
Now because $n - k \geq 4$, by letting $\delta \to 0$, we obtain
\begin{equation}
\int_{M} |\Hess\, u|^2 d\mu_g \leq 2 (\lambda - K) \int_M  |\nabla u|^2 d\mu_g < +\infty.
\end{equation}
\end{proof}

By using Lemma \ref{lem-eigenfunction-Hess-L2}, similar as in \cite{Mondello-thesis} on page 56 for eigenfunctions on stratified spaces, we prove

\begin{lem}\label{lem-eigenfunction-gradient-L-infinty}
Let $(M^n, g)$ as in Lemma \ref{lem-eigenfunction-Hess-L2}. For any $W^{1, 2}$-eigenfunction $u$ of $\Delta_g$, $|\nabla u| \in L^{\infty}(M)$.
\end{lem}
\begin{proof}
Let $u$ be a $W^{1, 2}$-eigenfunction of $\Delta$ with eigenvalue $\lambda$. Then $u \in C^{\infty}(\Omega)$. For simplicity of notation, we denote $v := |\nabla u|$. By (\ref{eqn-Bochner-inequality-regular-part-eigenfunction}), Kato's inequality  $|\nabla |\nabla u|| \leq |\Hess\, u|$, 
one can show that there exists a positive constant $c$ such that 
\begin{equation}\label{eqn-eigenfunction-gradient-inequality}
    -\Delta v \leq c v \ \ \text{in the barrier sense on} \ \ \Omega.
\end{equation}
Indeed, 
let $v_\epsilon= \sqrt{|\nabla u|^2 + \epsilon^2} -\epsilon$. 
By the Kato's inequality, one has
\begin{equation}\label{eqn-v-epsilon-Kato}
    |\nabla v_\epsilon|^2 = \frac{|\nabla u|^2 \cdot |\nabla |\nabla u||^2}{|\nabla u|^2 + \epsilon^2} \leq 
    \frac{|\nabla u|^2 \cdot |\Hess\, u|^2}{|\nabla u|^2 + \epsilon^2} 
   \leq |\Hess\, u|^2. 
\end{equation}
Moreover, putting $c:= \max \{\lambda-K, 1\}>0$, and using (\ref{eqn-Bochner-inequality-regular-part-eigenfunction}) and (\ref{eqn-v-epsilon-Kato}), one has
\begin{eqnarray*}
    \frac{1}{2} \Delta (v_\epsilon+\epsilon)^2 = \frac{1}{2} \Delta |\nabla u|^2
    & \geq & |\Hess\, u|^2 - (\lambda -K) |\nabla u|^2 \\
    & \geq & |\Hess\, u|^2 - c ((v_\epsilon+\epsilon)^2 - \epsilon^2) \\
    & \geq & |\nabla v_\epsilon|^2 - c (v_\epsilon+\epsilon)^2 \ \ \text{ on} \ \ \Omega.
\end{eqnarray*}
Thus, 
$$ (v_\epsilon+\epsilon) \Delta (v_\epsilon+\epsilon) + |\nabla v_\epsilon|^2 \geq  |\nabla v_\epsilon|^2 - c (v_\epsilon+\epsilon)^2.
$$
Consequently, $\Delta v_\epsilon \geq -c v_\epsilon -c\epsilon$ for any $\epsilon>0$. 

If $x\in \Omega$ and $v(x)\not=0$, then $v$ is smooth at $x$. The calculation above with $\epsilon=0$ shows that $\Delta v \geq -cv$. 

If $v(x_0)=0$, then $v_\epsilon(x_0)=0=v(x_0)$ and $v_\epsilon \leq v$. 
Hence, $\Delta v \geq -cv$ in the barrier sense at $x_0$. 

We now claim that the inequality in (\ref{eqn-eigenfunction-gradient-inequality}) holds in the weak sense on the whole manifold $M$. Then, with the help of the Sobolev inequality in Lemma \ref{lem-L2-strong-compactness},  a Moser iteration argument implies $v = |\nabla u| \in L^\infty(M)$. 

In the rest of the proof, we show that the inequality in (\ref{eqn-eigenfunction-gradient-inequality}) holds in the weak sense on $M$, that is, for any $\varphi \in W^{1, 2}(M)$, $\varphi \geq 0$, 
\begin{equation}\label{eqn-eigenfunction-gradient-weak-inequality}
    \int_M \langle \nabla v, \nabla \varphi \rangle d\mu_g \leq c \int_M v \varphi d\mu_g.
\end{equation}
Because $C^\infty(M)$ is dense in $W^{1, 2}(M)$, we only need to establish the inequality (\ref{eqn-eigenfunction-gradient-weak-inequality}) for $\varphi \in C^\infty(M)$, $\varphi \geq 0$. For such $\varphi$ and $\psi_\delta$ in Lemma \ref{lem-zero-capacity-property}, we multiply the inequality in (\ref{eqn-eigenfunction-gradient-inequality}) by $\psi_\delta \varphi$, and integrate over $M$. Then by integrating by parts and rearranging the inequality, we obtain
\begin{equation}\label{eqn-eigenfunction-gradient-weak-inequality-2}
    \int_M \psi_\delta \langle \nabla v, \nabla \varphi \rangle d\mu_g
    \leq
     c \int_{M} \psi_\delta \varphi v d\mu_g - \int_{M} \varphi \langle \nabla v, \nabla \psi_\delta \rangle d\mu_g.
\end{equation}
For the second term on the right hand side, because the smooth function $\varphi$ on the compact manifold $M$ is bounded, we have
\begin{equation*}
    \left| \int_M \varphi \langle \nabla v, \nabla \psi_\delta \rangle d\mu_g \right|
    \leq 
    C \left| \int_M \langle \nabla v, \nabla \psi_\delta \rangle d\mu_g \right|.
\end{equation*}
Then by applying Cauchy-Schwarz inequality and the Kato's inequality, we obtain
\begin{equation*}
    \left| \int_M \varphi \langle \nabla v, \nabla \psi_\delta \rangle d\mu_g \right|
    \leq
    C \|\nabla v\|_{L^2(M)} \|\nabla \psi_\delta\|_{L^2(M)} 
    \leq C \||\Hess\, u|\|_{L^2(M)} \|\nabla \psi_\delta\|_{L^2(M)} \to 0,
\end{equation*}
as $\delta \to 0$, since $\||\Hess\, u|\|_{L^2(M)}$ is finite by Lemma \ref{lem-eigenfunction-Hess-L2} and $\|\psi_\delta\|_{L^2(M)} \to 0$ by Lemma \ref{lem-zero-capacity-property}. Thus, by letting $\delta \to 0$ in (\ref{eqn-eigenfunction-gradient-weak-inequality-2}), we prove the inequality (\ref{eqn-eigenfunction-gradient-weak-inequality}) for nonnegative smooth function $\varphi$, so for nonnegative test functions $\varphi \in W^{1, 2}(M)$.
\end{proof}


\subsection{Proof of Theorems \ref{thm-smooth-extension}, \ref{thm-RCD} and \ref{thm-removable}} \label{subsect-proof-main-results}

With our work from the previous subsection, \thmref{thm-RCD} 
follows from \cite[Theorem 3.7]{Honda-AGMS-18},  
see also \cite{BKMR-AIF-21, Jiang-Sheng-Zhang2022
}. For completeness, we give a detailed proof here.
\begin{proof}[Proof of \thmref{thm-RCD}]
It suffices to check four conditions in \defref{defn-RCD}. Condition (1) follows from the compactness of $M$ and the uniform Euclidean property of $g$. Condition (2) is also clear. Condition (3) follows from \lemref{lem-Sobolev-to-Lipschitz}. 

It remains to check the BE($K, n$) condition in (4). Let $\Omega:= M \setminus S$, $\{u_i\}^{+\infty}_{i=1}$ be the orthonormal basis of $L^2(M)$ consisting of eigenfunctions of $\Delta$. For any $f \in W^{1, 2}(M)$ and $N \in \mathbb{N}$, let $f_N := \sum^{N}_{i=1} a_i u_i$, where $a_i := \int_{M} f u_i d\mu_g$. As before, $k=\dim S$,  and by assumption $n-k \geq 4$.

Bochner formula and $\Ric_g \geq K$ on $\Omega$ imply
\begin{equation}\label{eqn-Bochner-inequality-regular-part}
\frac{1}{2}\Delta |\nabla f_N|^2 \geq |\Hess f_N|^2 + \langle \nabla \Delta f_N, \nabla f_N \rangle +K |\nabla f_N|^2 \quad \text{on} \ \ \Omega.
\end{equation}

By Lemma \ref{lem-eigenfunction-Hess-L2} and Lemma \ref{lem-eigenfunction-gradient-L-infinty}, for any eigenfunction $u_i$, we have
\begin{equation*}
    \int_M |\nabla |\nabla u_i|^2|^2 d\mu_g
    = 4 \int_M |\nabla u_i|^2 |\nabla |\nabla u_i||^2 d\mu_g
    \leq 
    4 \int_M |\nabla u_i|^2 |\Hess\, u_i|^2 d\mu_g < +\infty.
\end{equation*}
Thus, $|\nabla u_i|^2 \in W^{1 ,2}(M)$ for all $i \in \mathbb{N}$, so $|\nabla f_N|^2 \in W^{1 ,2}(M)$. 
Also $|\Hess f_N| \in L^2(M)$ from Lemma \ref{lem-eigenfunction-Hess-L2}.


For any test function $\varphi \in L^\infty(M) \cap D(\Delta)$ with $\varphi \geq 0$ and $\Delta \varphi \in L^\infty(M)$, by multiplying the inequality \eqref{eqn-Bochner-inequality-regular-part} by $\varphi \psi_\delta$, and integration by parts, we have
\begin{eqnarray}\label{eqn-integral-Bochner-inequality-3}
    & & -\frac{1}{2}\int_M \langle \nabla(\varphi \psi_\delta), \nabla |\nabla f_N|^2 \rangle d\mu_g \nonumber \\
    & & \geq 
    \int_M \varphi \psi_\delta \left( |\Hess f_N|^2 + \langle \nabla \Delta f_N, \nabla f_N \rangle + K |\nabla f_N|^2 \right) d\mu_g.
\end{eqnarray}
By Lemma \ref{lem-zero-capacity-property}, $\psi_\delta \to 1$ in $L^2(M)$ and $\nabla \psi_\delta \rightharpoonup 0$ in $L^2(M)$ as $\delta \to 0$, and since $|\nabla f_N|^2 \in W^{1, 2}(M)$, we have
\begin{equation*}
    -\frac{1}{2}\int_M \langle \nabla(\varphi \psi_\delta), \nabla |\nabla f_N|^2 \rangle
    \to 
    -\frac{1}{2} \int_M \langle \nabla \varphi, \nabla |\nabla f_N|^2 \rangle d\mu_g
    =
    \frac{1}{2} \int_M |\nabla f_N|^2 \Delta \varphi d\mu_g,
\end{equation*}
as $\delta \to 0$.
Moreover, by the dominated convergence theorem, we have
\begin{eqnarray*}
    & & \int_M \varphi \psi_\delta \left( |\Hess f_N|^2 d\mu_g + \langle \nabla \Delta f_N, \nabla f_N \rangle + K |\nabla f_N|^2  \right) d\mu_g \\
    & \to &
    \int_M \varphi \left( |\Hess f_N|^2 + \langle \nabla \Delta f_N, \nabla f_N \rangle + K |\nabla f_N|^2 \right) d\mu_g.
\end{eqnarray*}
Thus, by letting $\delta \to 0$ in \eqref{eqn-integral-Bochner-inequality-3}, we have
\begin{eqnarray*}
   \frac{1}{2}\int_M |\nabla f_N|^2 \Delta \varphi d\mu_g
    & \geq  &
    \int_M \varphi \left( |\Hess f_N|^2 + \langle \nabla \Delta f_N, \nabla f_N \rangle + K|\nabla f_N|^2 \right) d\mu_g \\
    & \geq & 
    \int_{M} \varphi \left( \frac{(\Delta f_N)^2}{n} + \langle \nabla \Delta f_N, \nabla f_N \rangle + K|\nabla f_N|^2 \right) d\mu_g.
\end{eqnarray*}
Finally, letting $N \to \infty$, we obtain that 
\begin{equation*}
     \frac{1}{2}\int_M |\nabla f|^2 \Delta \varphi d\mu_g
       \geq 
    \int_{M} \varphi \left( \frac{(\Delta f)^2}{n} + \langle \nabla \Delta f, \nabla f \rangle + K |\nabla f|^2 \right) d\mu_g 
\end{equation*}
for any test function $\varphi \in L^\infty(M) \cap D(\Delta)$ with $\varphi \geq 0$ and $\Delta \varphi \in L^\infty(M)$.
This completes the proof.
\end{proof}

\thmref{thm-RCD} provides the crucial ingredient to the proof of \thmref{thm-smooth-extension}.
\begin{proof}[Proof of \thmref{thm-smooth-extension}]
First of all, we note that, under the assumptions for $M_0$ in \thmref{thm-smooth-extension}, there is no uniformly Euclidean metric $g$ on $M := T^n \# M_0$ smooth away from a closed embedded submanifold $S$ with co-dimension $\geq \frac{n}{2} +1$, which has nonnegative scalar curvature on the smooth part and positive scalar curvature at a point. This non-existence result can be established by using the conformal blow up technique dating back to Schoen's final resolution of Yamabe problem \cite{Schoen-Yamabe-problem}, which reduces the problem to the non-existence results of Chodosh-Li \cite{Chodosh-Li-Annals} and Wang-Zhang \cite{Wang-Zhang-22} for complete metrics.  This was recently used in studies of uniformly Euclidean metric with positive scalar curvature by Li-Mantoulidis \cite{Li-Mantoulidis}, Kazaras \cite{Kazaras-MAAN-19}, Wang-Xie \cite{WX2024} and others, and also in studies of positive mass theorem and positive scalar curvature with conical singularity by Dai-Sun-Wang \cite{DSW-PMT-nonspin, DSW-PSC-conical}. Since the method carries over in a similar way, we refer to these references for details. Roughly speaking, one uses the Green's function of the conformal Laplacian but with the scalar curvature suitably truncated to conformally blow up the singular metric. This produces a complete metric with positive scalar curvature on $(T^n \# M_0) \setminus S$, leading to a contradiction with non-existence results of Chodosh-Li \cite{Chodosh-Li-Annals} and Wang-Zhang \cite{Wang-Zhang-22}. We only emphasize a few things here. Firstly, for uniformly Euclidean metrics, the Green's function can be solved by Littman-Stampacchia-Weinberger \cite{LSW63} and Schoen-Yau \cite{SY-MM-79}. Also, the condition that the co-dmension of $S$ $\geq \frac{n}{2}+1$ guarantees the completeness of the blow up metric, see Proposition 2.4 in Wang-Xie \cite{WX2024}. And finally, the assumption on the location and size of the singular set, Condition 2), guarantees that $M\setminus S$ can again be written as $T^n \# M_1$ for suitable $M_1$, and hence the results of Chodosh-Li \cite{Chodosh-Li-Annals} and Wang-Zhang \cite{Wang-Zhang-22} can be applied. 

Then by Theorem B in Kazdan \cite{Kazdan-82}, the metric $g$ must be Ricci flat. Therefore, \thmref{thm-RCD} implies that $(T^n \# M_0, d_g, d\mu_g)$ is a ${\rm RCD}(0, n)$ space. Consequently, we conclude that $(T^n \# M_0, d_g)$ is isometric to a flat torus $T^n$ by Corollary 1.4 in Mondino-Wei \cite{Mondino-Wei-19}, and in particular, $g$ extends smoothly to $T^n \# M_0$ (possibly with a change of differential structure). This completes the proof.
\end{proof}

\begin{proof}[Proof of \thmref{thm-removable}]
By Theorem \ref{thm-RCD}, $(M, g)$ is ${\rm RCD}(0, n)$ and the splitting theorem of Gigli \cite{Gigli-splitting} applies. Since $\pi_1(M)$ is infinite, the universal cover $\tilde{M}$ splits as $\mathbb{R}^k \times X$ for some $1\leq k \leq n$ and $X$ compact simply connected metric measure space. Now if $\tilde{X}$ has singularity, it will give rise to non-isolated singularity for $M$. Since $(M, g)$ has only isolated singularity, $X$ must be smooth and our result follows, see \cite[Theorem 9.2]{CheegerGromoll72} and \cite{CheegerGromoll71} for more details. 
\end{proof}

We will end this section with several examples. 
\newline

{\bf Example 1}: This example is motivated by Cecchini-Frenck-Zeidler \cite{Cecchini2024}; see \cite{Cecchini2024} for related discussion. Let $M=T^n \# \Sigma$, where $\Sigma$ is an exotic sphere of dimension $n$. Consider $g$ a uniformly Euclidean metric on $M$ with an isolated singularity at a point in $\Sigma$ and nonnegative scalar curvature on the smooth part.
Then the metric measure space $(M, g)$ induced by $g$ does not remember the original smooth structure, and hence it is isometric to the flat torus $T^n$, which can be thought as the connected sum of the torus with the standard sphere. 
\newline

{\bf Example 2}: Let $M=T^n$ and $T^{k}\subset M$ a linear sub-torus where $n-k \geq 4$. If $g$ is a uniformly Euclidean metric on $M$ with singular set $S=T^{k}$ and nonnegative scalar curvature on the smooth part. Then by Gromov-Lawson \cite{Gromov-Lawson1980-simply}, $M\setminus S$ is $\Lambda^2$-enlargeable, and hence cannot carry complete metrics of positive scalar metric. It follows by the conformal blowup method as in the proof of Theorem \ref{thm-smooth-extension}, $g$ must be Ricci flat on the smooth part. Then by Theorem \ref{thm-removable}, $g$ extends smoothly over $S$. This example illustrates the situation where our result still applies even if the singular set is not contained in a topological ball. See \cite[Example 6.9]{Gromov-Lawson1980-simply} for more examples.
\newline

{\bf Example 3}: Consider a closed Schoen-Yau-Schick (SYS) manifold $M$ of dimension $4 \leq n \leq 7$ (here we use the definition given by Shi-Wang-Wu-Zhu \cite{SWWZ2024}). By the very recent work of Shi-Wang-Wu-Zhu \cite[Corollary 1.9]{SWWZ2024}, 
a uniformly Euclidean metric with isolated singularity on $M$ and nonnegative scalar curvature on the smooth part must be Ricci flat on the smooth part. On the other hand, 
a SYS manifold has infinite fundamental group. By Theorem \ref{thm-removable}, the metric extends smoothly, in other words, isolated singularity is removable.
Similarly, by Cecchini-Frenck-Zeidler \cite[Theorem C and Remark 1.9]{Cecchini2024}, a uniformly Euclidean metric with isolated singularity on a closed enlargeable manifold must be Ricci flat on the smooth part. Then because a closed enlargeable manifold has infinite fundamental group, again the singularity is removable.


\appendix

\section{Proof of \thmref{gradi}}\label{sect-appendix}

 The main ingredients are   Bochner's formula, the maximal principle, cut-off functions and the Laplacian comparison.  
  \begin{proof}[Proof of Theorem~\ref{gradi}]: Let $h = \log u$. Then $\nabla h =
\frac{\nabla u}{u}, \Delta h = \frac{\Delta u}{u} - \left|
\frac{\nabla u}{u} \right|^2 = \frac{K(u)}{u} - |\nabla h|^2$. By the Bochner
formula we have 
\ba \frac 12
\Delta |\nabla h|^2 & = & |\Hess \, h|^2 + \lp \nabla h, \nabla
(\Delta h) \rp
+ \Ric (\nabla h, \nabla h)   \nonumber \\
& \ge &  |\Hess \, h|^2  + \lp \nabla h,   \nabla \tfrac{K(u)}{u} \rp- \lp \nabla h, \nabla (|\nabla h|^2) \rp
-(n-1) H^2 |\nabla h|^2.  \label{bochner-h} \ea 
For the Hessian term
one could use the Schwarz inequality \be  \label{schwarz}
|\Hess \, h|^2 \ge
\frac{(\Delta h)^2}{n}.
\ee  Indeed, when $H = 0$, this estimate is enough. For $H>0$, using this estimate one would get $(n(n-1))^{1/2}$ instead of $n-1$ for the coefficient of $H$ in (\ref{gradient}). To get the best constant for $H >  0$, note that (\ref{schwarz}) is only optimal when Hessian at all
directions are same. Harmonic functions in the model spaces are
radial functions, so their Hessian along the radial direction would
be different from the spherical directions. Therefore one computes the norm by separating the radial
direction and uses Schwarz inequality in the spherical directions. Let $\{e_i\}$ be an orthonormal basis with $e_1 =
\frac{\nabla h}{|\nabla h|}$ (we only need to prove the case $|\nabla h| \not =0$), the potential radial direction, denote
$h_{ij} = \Hess \, h (e_i, e_j)$. We compute, 
\begin{eqnarray*}
|\Hess \, h|^2 & = &   h_{11}^2 + 2\sum_{j=2}^n h_{1j}^2 + \sum_{i,j\ge 2} h_{ij}^2 \nonumber \\  
& \ge & h_{11}^2 + 2\sum_{j=2}^n h_{1j}^2 + \frac{(\Delta h -h_{11})^2} {n-1}  \nonumber \\
& = & h_{11}^2 + 2\sum_{j=2}^n h_{1j}^2 + \frac{(|\nabla h|^2 +h_{11} -\tfrac{K(u)}{u})^2} {n-1} \label{eqn-Laplace-gradient}  \\
& \ge &  \frac{n}{n-1} \sum_{j=1}^n h_{1j}^2 + \frac{1}{n-1} |\nabla h|^4  + \frac{2}{n-1} h_{11} |\nabla h|^2  - \frac{2}{n-1} \tfrac{K(u)}{u} \left(h_{11} + |\nabla h|^2\right) . \nonumber
\end{eqnarray*}
Now  $h_{1j} = \frac{1}{ |\nabla h|} \lp \nabla_{e_j} \nabla h, \nabla h \rp =
\frac{1}{2 |\nabla h|} e_j ( | \nabla h|^2 )$ and $h_{11} =\frac{1}{2 |\nabla h|^2} \lp \nabla | \nabla h|^2, \nabla h \rp$. Therefore
\be  \label{Hess}
|\Hess \, h|^2  \ge    \frac{n}{4(n-1)} \frac{|\nabla | \nabla h|^2|^2}{ | \nabla h|^2} + \frac{|\nabla h|^4  + \lp \nabla | \nabla h|^2, \nabla h \rp}{n-1} - \frac{2}{n-1} \tfrac{K(u)}{u} \left(h_{11} + |\nabla h|^2\right). 
\ee
Note that when $K(u) =0$, i.e. $u$ is harmonic, then 
 equality holds if and only if $\Hess\, h$ are same on the level set of $h$, i.e.
\[
\Hess\, h = - \frac{| \nabla h|^2}{n-1} \left( g - \frac{1}{| \nabla h|^2} dh \otimes dh \right).
\]
By plugging (\ref{Hess}) into (\ref{bochner-h}) 
we get
\ba
\frac 12 \Delta
|\nabla h|^2  & \ge &   \frac{n}{4(n-1)} \frac{|\nabla | \nabla h|^2|^2}{ | \nabla h|^2} + \frac{|\nabla h|^4}{n-1}-  \frac{n-2 + \tfrac{K(u)}{u  |\nabla h|^2}}{n-1} \lp \nabla | \nabla h|^2, \nabla h \rp   \nonumber \\  && +  (K' - \tfrac{n+1}{n-1} \tfrac{K}{u}) | \nabla h|^2   -(n-1) H^2 |\nabla
h|^2.  \label{Delta-h}
\ea 
If $|\nabla h|^2$ achieves a local maximum inside $B(q,2R)$
then we are done. Assume $|\nabla h|(q_0)$ is the maximum for some
$q_0 \in B(q, 2R)$, then $\nabla |\nabla h|^2 (q_0) = 0,  \Delta
|\nabla h|^2 (q_0) \le 0$. Plug these into (\ref{Delta-h}) gives\[
0 \ge \frac{|\nabla h|^4}{n-1} +  (K' - \tfrac{n+1}{n-1} \tfrac{K}{u}) | \nabla h|^2  -(n-1) H^2 |\nabla h|^2.\] Hence
$|\nabla h|^2 \le (n-1)^2 H^2 + (n-1) \sup_{B(q,2R)} \left|K'- \tfrac{n+1}{n-1} \tfrac{K}{u}\right|$.

In general the maximum could occur at the boundary and one has to use a cut-off function to force the maximum is achieved in the interior. Let $f: [0,2R] \ra [0,1]$ be a smooth function with \ba
f|_{[0,R]} \equiv 1, \ \mbox{supp}f \subset [0, 2R),  \\
-cR^{-1} f^{1/2} \le f' \le 0,   \label{f'} \\
|f''| \le cR^{-2},  \label{f''} \ea where $c>0$ is a universal
constant. Let $\phi: \overline{B(q,2R)} \ra  [0,1]$ with $\phi(x)
= f (r(x))$, where $r(x)= d(x,q)$ is the distance function. Set $G
= \phi |\nabla h|^2$. Then $G$ is nonnegative on $M$ and has
compact support in $B(q,2R)$.  Therefore it achieves its maximum
at some point $q_0 \in B(q, 2R)$. We can assume $r(q_0) \in[R,
2R)$ since if $q_0 \in B(q, R)$, then $|\nabla h|^2$ achieves
maximal at $q_0$ on $B(q,R)$ and previous argument applies.

If $q_0$ is not a cut point of $q$ then $\phi$ is smooth at $q_0$
and we have \be \Delta G(q_0) \le 0, \ \ \ \nabla G (q_0) = 0.
\label{max-q0} \ee
At the smooth point of $r$, $\nabla G = \nabla \phi |\nabla h|^2 + \phi \nabla |\nabla h|^2$, $$\Delta G =\Delta  \phi |\nabla h|^2+ 2 \lp \nabla \phi, \nabla |\nabla h|^2 \rp +  \phi \Delta |\nabla h|^2.$$
Using (\ref{Delta-h}), (\ref{max-q0}) and express $|\nabla h|^2, \nabla |\nabla h|^2$ in terms of $G, \nabla G$ we get, At the maximal point $q_0$ of $G$,

\ba  0\ge  \Delta G (q_0)  &  \ge &  \frac{\Delta \phi}{\phi} G  -2  \frac{|\nabla \phi|^2}{\phi^2}G + \frac{n}{2(n-1)} \frac{|\nabla \phi|^2}{\phi^2}G  + \frac{2}{n-1} \frac{G^2}{\phi} \nonumber  \\
& &+ \frac{2(n-2 + \tfrac{K(u)}{u|\nabla h|^2})}{n-1} \lp \nabla h, \nabla \phi \rp
\frac{G}{\phi} +  2(K' - \tfrac{n+1}{n-1} \tfrac{K}{u}) G -2(n-1)H^2 G. \label{G-ineq} \ea
Since $\phi (x)  = f(r(x))$ is a radial function, $|\nabla \phi|
= |f'|$,  by (\ref{f'}),
\[
\lp \nabla h, \nabla \phi \rp \ge - |\nabla h| |\nabla \phi| =
-G^{\frac 12} \frac{|\nabla \phi|}{\phi^{1/2}} \ge - G^{\frac 12}
\frac{c}{R}.\] Also $\Delta \phi = f' \Delta r + f''$. Since $f'
\le 0$, by the Laplacian comparison theorem  and
(\ref{f''})
$$ \Delta \phi \ge f' \Delta_H r - cR^{-2},$$ where $\Delta_H r =
(n-1) H \coth (Hr)$, which is  $\le (n-1) H \coth (HR)$ on $[R,
2R]$. Hence
\ban
\Delta \phi & \ge & - cR^{-1} \left( (n-1)H \coth (HR) + R^{-1}
\right) \\
& \ge & - cR^{-1}  \left[ (n-1)(2R^{-1} + 4H) + R^{-1} \right]
\ean
Multiply (\ref{G-ineq}) by $\frac{(n-1)\phi}{G}$ and
plug
these in, we get 
\ban 0 & \ge&  2 G - 2(n-2 + |\tfrac{K(u)}{u|\nabla h|^2}|) \frac{c}{R} G^{\frac
	12}  \\ & & - \frac{c}{R} \left(
\frac{(n-1)(2n-1)}{R} + 4(n-1)^2 H +\left(
\frac{3n}{2}-2 \right)\frac{c}{R} \right) \\
& & - 2(n-1) \left| K' -\tfrac{n+1}{n-1} \tfrac{K}{u}\right| - 2(n-1)^2 H^2. 
\ean
Now if $|\nabla h|^2 \le \tfrac{|K(u)|}{u}$, we are done. So we can assume $-\tfrac{|K(u)|}{u|\nabla h|^2} \ge -1$. In this case we have 
\ban 
0 & \ge &  2 G - 2(n-1) \frac{c}{R} G^{\frac
	12}   - \frac{c}{R} \left(
\frac{(n-1)(2n-1)}{R} + 4(n-1)^2 H +\left(
\frac{3n}{2}-2 \right) \frac{c}{R} \right) \\
& & - 2(n-1)  \left| K' -\tfrac{n+1}{n-1} \tfrac{K}{u}\right| - 2(n-1)^2 H^2. 
\ean
Solving this
quadratic inequality gives  
\begin{equation*} 
(G(q_0))^{\frac 12} \le (n-1)H +
C_1(n) R^{-1} + C_2(n,H) R^{-1/2} + \left((n-1) \left| K' -\tfrac{n+1}{n-1} \tfrac{K}{u}\right|\right)^{1/2} .  
\end{equation*}
Therefore
\ban
& & \sup_{B(q,R)}|\nabla h| 
 =  \sup_{B(q, R)} G^{1/2}  
 \le   \sup_{B(q, 2R)} G^{1/2} \\
& & \quad \le (n-1) H  + C_1(n) R^{-1} + C_2(n,H) R^{-1/2} +\left[(n-1)  \left| K' -\tfrac{n+1}{n-1} \tfrac{K}{u}\right|\right]^{1/2}.
\ean
If $q_0$ is in the cut locus of $q$, use the upper barrier $r_{q_0, \epsilon}(x)$ for $r(x)$ and let $\epsilon \ra 0$ gives the same estimate.
\end{proof}

\bibliographystyle{plain}
\bibliography{DSW_PSC.bib}

\end{document}